\documentclass[11pt]{amsart}
\usepackage{amsmath}
\usepackage{tikz}

\theoremstyle{plain}
\newtheorem{theorem}{Theorem}[section]
\newtheorem{lemma}[theorem]{Lemma}

\newtheorem{cor}[theorem]{Corollary}

\newtheorem{lem}{Lemma}[section]
\theoremstyle{definition}
\newtheorem{remark}[theorem]{Remark}

\numberwithin{equation}{section}

\def\be{\begin{equation}}
\def\ee{\end{equation}}

\begin{title}[Isometric embedding under the graph setting]{
Isometric embedding with nonnegative Gauss curvature under the graph setting}
\author {Xumin Jiang}
\address{Department of Mathematics\\
Rutgers University\\
New Brunswick, NJ 08901}  \email{xj60@math.rutgers.edu}

\date{}

\end{title}

\begin{document}
\maketitle

\begin{abstract}
We study the regularity of the isometric embedding $X:(B(O, r), g) \rightarrow (\mathbb{R}^3, g_{can})$ of a 2-ball with nonnegatively curved $C^4$ metric into $\mathbb{R}^3$. Under the assumption that $X$ can be expressed in the graph form, we show $X \in C^{2,1}$ near $P$, which is optimal by Iaia's example.
\end{abstract}

\section{Introduction}
Weyl posted the following problem in 1916 \cite{Weyl}: Consider a positively curved 2-sphere $(S^2, g)$. Does there exist a global $C^2$ isometric embedding $X:(S^2, g) \rightarrow (\mathbb{R}^3, g_{can})$, where $g_{can}$ is the standard flat metric on $\mathbb{R}^3$? Weyl himself suggested the continuity method to solve this problem and obtained a priori estimates up to the second derivatives. Lewy \cite{Lewy} solved the problem under the assumption the $g$ is analytic. In 1953, Nirenberg  \cite{Nirenberg}  solved the Weyl problem under the mild smoothness assumption that the metric is $C^4$.

P. Guan and Y.Y. Li \cite{GuanLi} considered the question that if the Gauss curvature of the metric $g$ is nonnegative, whether does $(S^2, g)$ still have a smooth isometric embedding? They proved in \cite{GuanLi} that for any $C^4$ metric $g$ on $S^2$, there is a global $C^{1,1}$ isometric embedding into $\mathbb{R}^3$. Examples in Iaia \cite{Iaia} show that for some analytic metrics with positive Gauss curvature on $S^2$ except at one point, there exists only a $C^{2,1}$ but not a $C^3$ global isometric embedding into $\mathbb{R}^3$. 

Then a natural question, posted in \cite{GuanLi}, is that if a smooth metric $g$ on $S^2$ has nonnegative Gauss curvature, whether does it have a $C^{2, \alpha}$, for some $0<\alpha<1$, or even a $C^{2,1}$ global isometric embedding? To study this problem, we can look at the degenerate Monge-Amp\`ere equation
\begin{align}\label{eq-MA}
\det(D^2 u)= k
\end{align}
where $k(x,y)\geq 0$, in $B_r(O)$ for small $r>0$. Guan \cite{Guan} considered the case $k\in C^\infty(B_r(O))$, and
\begin{align}\label{eq-KBound}
\frac{1}{A}(x^{2n}+ By^{2m}) \leq k(x,y) \leq A(x^{2n}+ By^{2m}),
\end{align}
for some $A>0, B\geq 0$ and positive integers $n\leq m$. It's shown in \cite{Guan} that a $C^{1,1}$ solution $u$ of \eqref{eq-MA} is smooth near the origin if \eqref{eq-KBound} holds, and if, additionally 
\begin{align}\label{eq-UXX}
u_{xx}\geq C_0 >0.
\end{align}
Guan and Sawyer \cite{GuanSawyer} improved this result by replacing \eqref{eq-UXX} by a weaker condition $\Delta u\geq C_0 >0.$

Daskalopoulos and Savin \cite{DaskSavin} considered \eqref{eq-MA} in the case that $k$ is radial. It's shown that in \cite{DaskSavin} that if $k(x,y)= (x^2+y^2)^\frac{\delta}{2}$ for some $\delta>0$, then $u\in C^{2,\epsilon}$ for a small $\epsilon$ which depends on $\delta$.

\medskip
In this paper, we  consider a $C^{1,1}$  isometric embedding $X:(B(O, r), g) \rightarrow (\mathbb{R}^3, g_{can})$, where $B(O, r)$ is a ball  in $\mathbb{R}^2$, centred at the origin with radius $r$, and $g$ is a $C^4$ metric with Gauss curvature $k\geq 0$. We regard the image of $X$,  as the graph of a function $u$. In fact, we can assume the isometric embedding is of form
\begin{align}\label{eq-GeneEmb}
X: (x, y) &\longmapsto (\alpha(x, y), \beta(x, y), u(x, y)),
\end{align}
where $k=0$ at $O$. Under normalization, we may assume
\begin{align}\label{eq-GraphU}
u(0,0)=0, u_x(0,0)=u_y(0,0)=0.
\end{align}
Notice \eqref{eq-GraphU} implies that the $\alpha\beta$-plane in $\mathbb{R}^3$, is tangent to the image $X(S^2)$ at $X(O)$.

An example is that 
\begin{align*}
\alpha=x, \beta=y, u=r^3= (x^2+y^2)^\frac{3}{2}.
\end{align*}
Then $g= dx^2+ dy^2 + du^2$ is smooth in $x,y$, $k(x, y)= 18 (x^2+y^2) >0$ except at the origin, but the embedding is only $C^{2,1}$.

First we have the following theorem when $\alpha (x, y) \equiv x, \beta (x, y)\equiv y$, and the Gauss curvature  only degenerates at a single point.
\begin{theorem}\label{thm-Main}
Assume that we have a $C^4$ metrc $g$ on a ball $B(O, r) \subseteq \mathbb{R}^2$, for some $r>0$, and that the Gauss curvature $k>0$ in $B(O, r)-\{O\}$. Assume  that a $C^{1,1}$ isometric embedding $X:(B(O, r), g) \rightarrow (\mathbb{R}^3, g_{can})$ is of form
\begin{align}\label{eq-GImb}
X: (x, y) &\longmapsto (x, y, u(x, y)),
\end{align}
under local coordinates $x,y$  such that \eqref{eq-GraphU} holds.
Then $X \in C^{2,1}(B(O, r))$.
\end{theorem}
Here we only need the sign of the Gauss curvature $k$.
By the example $u=r^3$, we see that the $C^{2,1}$ smoothness of $X$ is optimal.

Enlightened by Guan and Sawyer \cite{GuanSawyer}, if $\Delta u$, or the mean curvature $H$, has a uniform positive lower bound near but not necessarily at the origin, we have
\begin{cor}\label{cor-Mean}
Assume the same assumptions as  in Theorem \ref{thm-Main}. In addition, we assume that $g\in C^\infty$, and $\Delta u = u_{xx} + u_{yy} > C_0 > 0$ for some constant $C_0$, around $O$ but not necessarily at $O$, then $X\in C^\infty(B_g(P, r))$.
\end{cor}

For the Monge-Amp\`ere equation \eqref{eq-MA}, in the case that $u$ is radial, 
we have the following corollary showig $u\in C^{2,1}$, which is optimal by the example $u=r^3$. This result is expected to be true. We list it here as a quick corollary of lemmas in Section \ref{sec-OneDim}.
\begin{cor}\label{cor-MA}
Assume that a $C^{1,1}$ convex function $u$ satisfies \eqref{eq-MA} in $B(O, \rho)$, the ball of radius $\rho$ centered at the origin, and $k\geq 0$ in $B(O, \rho)-O$. In addition, we assume that $u=\Phi(r)$ for some function $\Phi$, where $r=\sqrt{x^2+y^2}$ , and $k\in C^3(B(O, \rho))$. Then $u\in C^{2,1}(B(O, \rho))$.
\end{cor}
Here $k$ could vanish at infinite order at $r=0$. We see that $\Phi_r$ is the square root of a $C^4$ function.

Under the general nonnegative Gauss curvature, Pogorelov's counterexample in \cite{Pogorelov} shows that a $C^{2,1}$ metric with nonnegative Gauss curvature may not have a $C^{2}$ isometric embedding. However, given a $C^4$ metric, under the graph setting, our result is positive.
\begin{theorem}\label{thm-Main1}
Assume that we have a $C^4$ metrc $g$ on a ball $B(O, r) \subseteq \mathbb{R}^2$, for some $r\geq 0$, with Gauss curvature $k\geq 0$. Assume that  a $C^{1,1}$ isometric embedding $X:(B(O, r), g) \rightarrow (\mathbb{R}^3, g_{can})$ is of form
\begin{align*}
X: (x, y) &\longmapsto (x, y, u(x, y)),
\end{align*}
under local coordinates $x,y$  such that the normalization \eqref{eq-GraphU} holds. If in addition, $u$ is (weakly) convex, then $X \in C^{2,1}(B(O, r^\prime))$, for any $r^\prime<r$. 
\end{theorem}

The paper is organized as follows. In Section \ref{sec-OneDim}, we will discuss the one dimensional model. In Section \ref{sec-TwoDim}, we will prove Theorem \ref{thm-Main}. In Section \ref{sec-Cor}, we prove Corollary \ref{cor-Mean}. In Section \ref{sec-NonNeg}, we prove Theorem \ref{thm-Main1}.

\bigskip
Many thanks to Yanyan Li for introducing me this problem and the whole project. Thanks to Zheng-Chao Han for helpful discussions.

\bigskip
\section{A model in  dimension one}\label{sec-OneDim}
In this section, we derive $C^{2,1}$ estimates of $u$ in an one dimension model.

Assume a nonnegative function $u=u(x) \in C^{1}(2I)$, where $I= [-1,1]$, and $u(0)=0$. In addition, we assume that $f=u_x^2$ is in $C^4(2I)$, and $f^\prime(x) x\geq 0$. The goal is to show $u\in C^{2,1}(I)$. The condition  that $f^\prime(x) x\geq 0 $ is necessary, since there is a nonnegative function 
\begin{eqnarray*}f(x)=
e^{-\frac{1}{x^2}} \sin^2 (\frac{1}{x}) + e^{-\frac{2}{x^2}}
\end{eqnarray*}
which is smooth, vanishes at infinite order at $x=0$, and $(\sqrt{f})^{\prime\prime}$ blows up when $x \rightarrow 0$. In fact, $\sqrt{f}\in C^{1,\alpha}$ for any $0<\alpha<1$, and $(\sqrt{f})^{\prime\prime}\leq \frac{C}{|x|^4}$ for some fixed $C$. $x^4 \sqrt{f}$ is a $C^{1,1}$ function. It shows that $|u^{\prime\prime\prime}|= | (\sqrt{f})^{\prime\prime}|$ is not bounded when $x$ tend to zero.

We have the following well known lemma,
\begin{lemma}\label{lem-grad}
 Assume $f \in C^2(2I)$, where $I=[-1, 1]$. $f\geq 0$ for $x  \in 2I$. Then for every $x\in I$, \begin{align*}
|f^\prime(x)| \leq \frac{3}{2}||f||_{C^2(2I)}^\frac{1}{2} f(x)^\frac{1}{2}.
\end{align*}
\end{lemma}
\begin{proof}
Assume first $||f||_{C^2(2I)=1}$. We only need to consider at $x$ where $f(x)>0$.
If $f(x)\geq 1$ at some $x\in I$, then
$|(\sqrt{f})^\prime(x)|= |\frac{f^\prime(x)}{2\sqrt{f(x)}}|\leq \frac{1}{2}||f||_{C^2(2I)}=\frac{1}{2}$. We assume $0<f(x)<1$ at some $x \in  I$, 
then
by the Taylor expansion, if $x+t\in I$,
\begin{align*}
f(x+t)= f(x) + f^\prime(x) t + \frac{f^{\prime\prime}(\tilde{x})}{2}t^2\geq 0,
\end{align*}
for some $\tilde{x}$ between $x$ and $x+t$.
Then
\begin{align*}
 f^\prime(x) t \geq -f(x)- \frac{f^{\prime\prime}(\tilde{x})}{2}t^2.
\end{align*}
When $ f^\prime(x)>0$, we set $t=  -f(x)^\frac{1}{2}$, and divide $t$ on both hand sides to derive,
\begin{align*}
0< f^\prime(x)\leq \frac{3}{2} f(x)^\frac{1}{2}.
\end{align*}
When $ f^\prime(x)<0$, we set $t= -f(x)^\frac{1}{2}$, and divide $t$ on both hand sides to derive,
\begin{align*}
0>f^\prime(x)\geq  -\frac{3}{2} f(x)^\frac{1}{2}.
\end{align*}
Notice the choice of $t$ is valid, since $|x+t|\leq |x|+|t|\leq |x|+\sqrt{f}<2$.

So we derived,
\begin{align*}
|f^\prime(x)| \leq \frac{3}{2} f(x)^\frac{1}{2}.
\end{align*}

If general, when $||f||_{C^2(2I)}\neq 1$, by a scaling, we see that
\begin{align*}
|f^\prime(x)| \leq \frac{3}{2}||f||_{C^2(2I)}^\frac{1}{2} f(x)^\frac{1}{2}.
\end{align*}
\end{proof}

The following is a standard interpolation lemma,
\begin{lemma}\label{lem-G}
Assume that $G(y)$ is a $C^4$ function defined on $[-1,0]$ such that $G(y)\geq 0$, and is non-decreasing. Then there exist universal constants $A, B$ such that
\begin{align*}
|G^\prime(0)|+|G^{\prime\prime}(0)|+ |G^{\prime\prime\prime}(0)| \leq AG(0)+B\max_{y\in [-1,0]} |G^{(4)}(y)|.
\end{align*} 
\end{lemma}
\begin{proof}
By the Taylor expansion,
\begin{align*}
G(-1)&= G(0)+ G^\prime(0) (-1) + \frac{G^{\prime\prime}(0)}{2}(-1)^2+\frac{G^{\prime\prime\prime}(0)}{6}(-1)^3+\frac{G^{(4)}(\xi_1)}{24}(-1)^4,\\
G(-\frac{1}{2})&= G(0)+ G^\prime(0) (-\frac{1}{2}) + \frac{G^{\prime\prime}(0)}{2}(-\frac{1}{2})^2+\frac{G^{\prime\prime\prime}(0)}{6}(-\frac{1}{2})^3+\frac{G^{(4)}(\xi_2)}{24}(-\frac{1}{2})^4,\\
G(-\frac{1}{4})&= G(0)+ G^\prime(0) (-\frac{1}{4}) + \frac{G^{\prime\prime}(0)}{2}(-\frac{1}{4})^2+\frac{G^{\prime\prime\prime}(0)}{6}(-\frac{1}{4})^3+\frac{G^{(4)}(\xi_3)}{24}(-\frac{1}{4})^4,\\
\end{align*}
for some $-1<\xi_1<0, -\frac{1}{2}<\xi_2< 0, -\frac{1}{4}< \xi_3<0$. Regard these as 3 linear equations in $G^\prime(0) ,G^{\prime\prime}(0) ,G^{\prime \prime \prime}(0) $, and solve them in terms of $G(-1)-G(0), G(-\frac{1}{2})-G(0), G(-\frac{1}{4})-G(0), G^{(4)} (\xi_1), G^{(4)} (\xi_2), G^{(4)} (\xi_3)$.
Since $G(x)$ is non-decreasing and nonnegative, we have $|G(-1)-G(0)|$, $|G(-\frac{1}{2})-G(0)|$, $|G(-\frac{1}{4})-G(0)| \leq G(0)$. Then the lemma follows.
\end{proof}

Next is the key Theorem in this section.
\begin{theorem}\label{thm-MainEstm}
Assume that  $u$ is a $C^{1}$ function such that $f=u_x^2 \in C^4(2I)$. In addition, assume that $u_x(0)=0$ and $f^\prime(x)x$ is nonnegative. 

 Then $u$ is  $C^5$ in $I -\{f=0\}$, and for every $x \in I -\{f=0\}$,
\begin{align}\label{eq-Main1D0}
| u_{xxx}(x)|\leq C ||f||_{C^4(2I)}^\frac{1}{2},
\end{align}
for some universal constant $C$.
\end{theorem}

\begin{proof}
At $x\in I- \{f=0\}$, $f(x)>0$. So $u_x= \sqrt{f}>0$ in a neighborhood of $x$, or $u_x= -\sqrt{f}<0$ in a neighborhood of $x$. In both cases, $f \in C^4(I)$ implies $u$ is $C^5$ near $x$.

First we assume $ ||f||_{C^4(2I)}=1$.
Our goal is to prove, for $x \in I-\{f=0\}$,
\begin{align}\label{eq-MainEq1D}
|u_{xxx}|= |(\sqrt{f})^{\prime\prime}|=|\frac{2ff^{\prime\prime}- (f^\prime)^2}{4f^\frac{3}{2}}|\leq  C,
\end{align} 
for some universal constant $C$.

Denote $g=\frac{f}{x^2}, h=\frac{f^\prime}{x}$. By Lemma \eqref{lem-FGC2}, $g, h \in C^2(2I)$ and the $C^2$ norms of $g, h$ are bounded by $C||f||_{C^4}=C$ for some universal constant $C$. Notice $g, h$ are nonnegative.
If we apply Lemma \ref{lem-grad} to $g$, then we derive, for $x \in I$,
\begin{align}
|xg^\prime|=
|\frac{x f^\prime-2f}{x^2}| \leq C f^\frac{1}{2}.\label{eq-GP1}
\end{align}
And applying Lemma \ref{lem-grad} to set $h$,  we derive, for $x \in I$,
\begin{align}\label{eq-HP1}
|h^\prime|=|\frac{xf^{\prime\prime}-f^\prime}{x^2}|\leq C\left(\frac{f^\prime}
{x}\right)^\frac{1}{2}=C(xg^\prime+2\frac{f}{x^2})^\frac{1}{2}\leq C( f^\frac{1}{4}+|x|^{-1} f^\frac{1}{2}). 
\end{align}

Case 1: Consider points $x\in I-\{f=0\}$, such that $f(x)\geq x^4$. We see that $|x| f^\frac{1}{2}\leq f^\frac{3}{4}, |x| f^\frac{1}{4}\leq f^\frac{1}{2}$. Then by \eqref{eq-GP1}, \eqref{eq-HP1},
\begin{align}
|x^2 g^\prime|=|f^\prime-\frac{2f}{x}|&\leq Cf^\frac{3}{4} \label{eq-GP},\\
|xh^\prime|=|f^{\prime\prime}-\frac{f^\prime}{x}|&\leq Cf^\frac{1}{2}
\label{eq-HP},
\end{align}
and by \eqref{eq-GP1}, \eqref{eq-GP}, \eqref{eq-HP},
\begin{align*}
|f  x h^\prime| &=f |xh^\prime| \leq Cf^\frac{3}{2},\\ 
x^4(g^\prime)^2 &= (x^2 g^\prime)^2 \leq Cf^\frac{3}{2},\\
|xf g^\prime| &= f|x g^\prime| \leq C f^\frac{3}{2}.
\end{align*}
which further implies $2ff^{\prime\prime}-(f^\prime)^2 =2fxh^\prime-x^4 (g^\prime)^2-2xfg^\prime$ is bounded by $Cf^\frac{3}{2}$. So \eqref{eq-MainEq1D} holds.

Case 2: Consider points $x\in I-\{f=0\}$, such that $f(x)\leq x^4$.
We are to prove at such  $x$,
\begin{align}\label{eq-DfOrder}
|f^\prime|\leq Cf^\frac{3}{4}, |f^{\prime\prime}|\leq Cf^\frac{1}{2}, |f^{(3)}|\leq C f^\frac{1}{4}, 
\end{align}
for some universal constant $C$,
which implies \eqref{eq-MainEq1D}.

If $0 <\epsilon \leq x$, Lemma \ref{lem-G} can be applied to the function $G(y)=f(x+\epsilon y)$ for $-1\leq y\leq 0$. Since 
\begin{align*}
G^\prime(0)&= \epsilon f^\prime (x),\quad G^{\prime\prime}(0) = \epsilon^2 f^{\prime\prime} (x), \\
 G^{\prime\prime\prime}(0) & = \epsilon^3 f^{\prime\prime\prime} (x), \quad
 \max_{y\in [-1, 0]} |G^{(4)}(y) |\leq  \epsilon^4 \max_{x\in I} |f^{(4)} (x)|,
\end{align*}
we derive
\begin{align}\label{eq-interpolation}
\epsilon |f^\prime(x)| +\epsilon^2 |f^{\prime\prime} (x) | + \epsilon^3 | f^{\prime\prime\prime} (x)| &\leq Af(x) +\epsilon^4 B \cdot \max_{x\in I} |f^{(4)}(x)|.
\end{align}
By setting $\epsilon= f(x)^\frac{1}{4}\leq x$,  \eqref{eq-interpolation} implies \eqref{eq-DfOrder}. If for some $x$, $x<0$ and  $f(x)\leq x^4$, we set $G(y)= f(x-\epsilon y)$ for $-1\leq y\leq 0$, then we derive \eqref{eq-DfOrder} in a similar way.

In sum, \eqref{eq-MainEq1D} is verified, and we derive \eqref{eq-Main1D0} by as a scaling.
\end{proof}

\begin{remark}
Under the assumption of Theorem \ref{thm-MainEstm} and $||f||_{C^N(2I)}\leq C$ for some fixed large $N$, in general $u$ does not have a uniform $C^{3,\alpha}$  estimate, for any $0<\alpha<1$, on a connected component of $I-\{f=0\}$. A counterexample is a family of functions 
\begin{align}\label{eq-CountEx}
u_s(x)=(x^2+s^2)^\frac{3}{2},
\end{align}
and $f=u_x^2= 9s^2 x^2+ 9x^4$. 
\end{remark}

Checking the behavior of $u_{xx}$ near the origin, we have,
\begin{theorem}\label{thm-C21}
Assume the same assumption as Theorem \ref{thm-MainEstm}, and in addition, $u_{xx}\geq 0$ for any $x\in 2I- \{f=0\}$.
 Then 
$
u\in C^{2,1}( I),
$ with \eqref{eq-Main1D0} holds.
\end{theorem}
\begin{proof}
We show
$u \in C^2(I)$, then the theorem follows from Theorem \ref{thm-MainEstm}.

First assume $f=0$ only at $x=0$.
Taking the Taylor expansion of $f(x)$ at $0$, if $f(x)= Mx^2 + R(x)$, for some   $M>0$, and $R(x)\in O(x^3)$, then for $x\neq 0$,
\begin{align*}
u_{xx}= sign(x) \cdot \frac{f_x}{2\sqrt{f}}= sign(x) \cdot \frac{2Mx+R_x}{2\sqrt{ Mx^2 + R}},
\end{align*}
which approaches $\sqrt{M}$ as $x$ tends to 0. Also we check
\begin{align*}
u_{xx}(0)= \lim_{x \rightarrow 0} \frac{u_x(x)}{x}= \lim_{x \rightarrow 0} u_{xx}(x)=\sqrt{M}
\end{align*}
by L'Hospital's Rule.

If $M=0$, then $f(x)=O(x^4)$. After a scaling, we assume $f(x)\leq x^4$. Then for $x$ near the origin,
\begin{align}\label{eq-Uxx0}
u_{xx}=sign(x)\cdot \frac{f_x}{\sqrt{f}} \leq \frac{C f^\frac{3}{4}}{\sqrt{f}}=C\sqrt{f},
\end{align}
by  \eqref{eq-DfOrder}.
So $u_{xx}(0)$ exists and equals 0.

Secondly, if $f=0$ at some $x_0\neq 0$, without loss of generality, assume $x_0>0$. Since $f^\prime(x) x\geq 0$, $f$ is non-decreasing as $x>0$. Then $f\equiv 0$, for $0\leq x\leq  x_0$. Denote $x_1 = \max\{x \in 2I: f(x)=0\}$. We only have to consider $u_{xx}$ at $x_1$ if $x_1 \in I$. We can use a new coordinate that translates $x_1$ to the origin, and apply an argument like \eqref{eq-Uxx0} to show $u_{xx}(x_1)=0$.
\end{proof}
Example $u=|x|^3$ shows that in general $u \notin C^3(\frac{1}{2} I)$ under the assumption of Theorem \ref{thm-C21}  even if $f\in C^\infty(I)$.

For the case $f>0$ in $2I$,  we see that $u$ is a $C^5$ function on $2I$.
\begin{cor}\label{cor-MainEstm}
Assume $u$ is a $C^{2}$ function such that $f=u_x^2 \in C^4(2I)$. In addition, $f>0, u_{xx}\geq 0$ in $(-2, 2)$. Then for every $x \in I$,
\begin{align}\label{eq-Main1D01}
| u_{xxx}(x)|\leq C ||f||_{C^4(2I)}^\frac{1}{2},
\end{align}
for some universal constant $C$.
\end{cor}
\begin{proof}
$u_{xx}\geq 0$ implies $f$ is non-increasing or non-decreasing on $(-2, 2)$, depending on the sign of $u_x(0)$.
Without loss of generality, we shift the origin to $-2$, and assume $f$ is non-decreasing on $(0,4)$.

Assume $||f||_{C^4(2I)}= 1$ first.
Our goal is to prove \eqref{eq-DfOrder} in $[1, 3]$. In fact, for any $x \in [1, 3]$, we can derive \eqref{eq-interpolation} for $\epsilon\in (0, x)$. We select $\epsilon=( \frac{1}{2} f(x))^\frac{1}{4} <  f(x)^\frac{1}{4} \leq 1 \leq  x$. The rest follows as Theorem \ref{thm-MainEstm}.
\end{proof}

For the applications in Section \ref{sec-TwoDim} and \ref{sec-NonNeg}, we need a scaling version of Theorem \ref{thm-MainEstm}, \ref{thm-C21}. Assume $u$ is a $C^{1,1}$ function such that $f=u_x^2 \in C^4(2sI)$ for some $0<s<1$. In addition, $u_x(0)=0$ and $f^\prime(x)x$ is nonnegative. In addition $u_{xx} \geq 0$ in $2sI-\{f=0\}$. We define
\begin{align*}
\bar{u}(x)=u(s^{-1} x),
\end{align*}
then $\bar{u}$ satisfies the assumption of Theorem \ref{thm-MainEstm}. We derive $\bar{u}\in C^2(I)$ and 
\begin{align*}
|\bar{u}_{xxx}| \leq C ||\bar{u}_x^2||^\frac{1}{2}_{C^4(2 I)},
\end{align*}
in $I-\{\bar{u}_x=0\}$,
implying $u \in C^2(sI)$, and 
\begin{align}\label{eq-C21S}
|u_{xxx}|\leq Cs^{-3} ||u_x^2||^\frac{1}{2}_{C^4(2s I)},
\end{align}
in $sI-\{u_x=0\}$. \eqref{eq-C21S} is also right under the assumption of Corollary \ref{cor-MainEstm} if we shrink the interval by multiplying the factor $s$.

\bigskip
\section{Two dimensional case with one singular point}\label{sec-TwoDim}
In this section, we prove Theorem \ref{thm-Main}. Maybe making  $r$ a bit smaller, we assume that the $C^{1,1}$ function $u$ in \eqref{eq-GImb} is defined in $rI \times rI$. In addition, we assume that $f=u_x^2>0$, $u_{x}$ exists and be positive, except at the origin. Here  $u_{xx}$ exists in $B(O,r) -\{O\}$, since the Gauss curvature $k>0$ except at the origin. Then by the classic theory of Monge-Amp\`ere equations, $g$ is $C^4$ implies that $u$ is $C^{3, \alpha}$, except at the origin, for any $\alpha \in (0,1)$. See Section 10.3 of \cite{Schulz}.

For $(x, y)\neq 0$, $\{u_x=0\}$ is locally a curve, since at any point except the origin, $u_{xx}(x,y)>0$, then we can solve out $x= A(y)$ as a function of $y$ from the equation $u_x(x,y)=0$. Furthermore, $\frac{d}{dy} A(y)=-\frac{u_{xy}(A(y), y)}{u_{xx}(A(y), y)}$, which is uniformly bounded when $y\in (\delta, r) \cup (-r, -\delta)$ for any $\delta>0$. In addition, for each $y$, we can only have at most one $x$, such that $u_x (x, y)=0$, since $u_x$ is strictly increasing. Though the gradient of $A(y)$ may blow up when $y$ approaches $0$, we  show that $\{u_x=0\}$ is a continuous curve.

\begin{lem} Assume that $u \in C^{1,1}(rI\times rI)$, $u_x(0,0)=0$ and $u_{x}$ is an increasing function in $x$ for any fixed $y$. Then
$\{u_x=0\}$ is a continuous curve near the origin.
\end{lem}
\begin{proof}
We only have to show $A(y)$ is continuous at $y=0$, i.e.
for any $\epsilon \in (0,1)$, we can find an $\delta>0$, such that $-\delta<y<\delta$ implies $-\epsilon < A(y) < \epsilon$. 

On the segment $\{(x,y): -r\leq x \leq r, y=0\}$, $u_x$ is increasing. Assume $u_x(\epsilon, 0)>\eta$, and $u_x(-\epsilon, 0)<-\eta$, for some $\eta>0$. Then there is an $\delta>0$, such that when $|y|<\delta$, 
\begin{align*}
u_x(\epsilon, y)>0, \quad u_x(-\epsilon, y)<0.
\end{align*}
The choice of $\delta$ depends on $\eta, ||u||_{C^{1,1}}$. Hence, for any fixed $y\in(-\delta, \delta)$, the zero of $u_x$  must be unique and the value of $x$ lies in $(-\epsilon, \epsilon)$, by the assumpition that $u_{x}$ is an increasing function in $x$ for any fixed $y$. So we derive $-\epsilon< A(y)<\epsilon$ as $|y|<\delta$.
\end{proof}

Now we prove Theorem \ref{thm-Main}.
\begin{proof}
By the assumption, the metric
\begin{align*}
g=d x^2+ dy^2 + du^2 =(1+u_x^2) d x^2 + u_x u_y dx dy + u_x u_y dy dx+(1 + u_y^2) dy^2
\end{align*}
is $C^4$, which implies that $u_x^2, u_y^2, u_x u_y \in C^4$. Hence $u_z^2 \in C^4$ for any $z= lx+my$, where $l, m$ are fixed numbers in $\mathbb{R}$.
At points except the origin, $k>0$. Then the classic theory of Monge-Amp\`ere equation shows that $u\in C^{2,\alpha}$ at any point $(x,y)$ away from the origin. We get $u_{zz}>0$ except at the origin.

Now fix a small $\epsilon<<r$. Assume $u_x(-\epsilon, 0) <-2 \eta_1, u_x(\epsilon, 0)>2 \eta_1$ for  some $\eta_1>0$. Then
$u_x(-\epsilon, y) <-\eta_1, u_x(\epsilon, y)>\eta_1$, if $|y|<\delta_1 = \frac{\eta_1}{||u||_{C^{1,1}(rI \times rI)}}$. We set 
\begin{align*}
s= \max_{|y|<\delta_1} \{ \epsilon-A(y), \epsilon+A(y)\},
\end{align*}
which has bound $\epsilon\leq s \leq 2\epsilon$.
Then the interval $[A(y)-2s,  A(y)+2s]\in I$ since $\epsilon$ is small. And for any $y\in (-\delta_1, \delta_1)$,
\begin{align*}
f(x, y)= u_x^2(x, y)> \eta_1^2,
\end{align*}
if $x= A(y)-2s$ or $A(y)+2s$. So by \eqref{eq-C21S}, for any $y\in (-\delta_1,\delta_1)$, $u$ has $C^{2,1}$ estimates for $x \in [A(y)-s,  A(y)+s]$, which depends only on $s, ||g||_{C^4(rI\times rI)}$. For points $(x, y) \in (rI  - [A(y)-s,  A(y)+s]) \times \delta_1 I$, 
\begin{align*}
f(x, y)= u_x^2(x, y)> \eta_1^2,
\end{align*}
and hence by \eqref{eq-MainEq1D},
\begin{align*}
|u_{xxx}| = |\frac{2ff^{\prime\prime}- (f^\prime)^2}{4f^\frac{3}{2}}|\leq  \eta_1^{-3} ||f||_{C^2(rI \times rI)}^2.
\end{align*}

Now $u \in C^3(rI \times rI - (0,0))$, since the metric $g \in C^4$. In $rI \times \delta_1 I$, $u$ has $C^{2,1}$ estimates in $x$, which depends only on $\eta_1, \epsilon, ||g||_{C^4}$, that is,
\begin{align*}
||u_{xxx}(x, y)||\leq C(\eta_1, \epsilon, ||g||_{C^4}),
\end{align*}
for $(x, y) \in rI \times \delta_1 I- (0,0) $.

Similarly, switching $x$ and $y$, we can find $\eta_2$ such that
\begin{align*}
||u_{yyy}(x, y)||\leq C(\epsilon, \eta_2, ||g||_{C^4}),
\end{align*}
for $(x, y) \in \delta_2 I \times r I- (0,0) $, where $\delta_2=\frac{\eta_2}{{||u||_{C^{1,1}(rI \times rI)}}}$.

Apply the same argument to coordinates $z=\frac{x+y}{2}$ and $w=\frac{x-y}{2}$, we have
\begin{align*}
u_{zzz}&= u_{xxx}+ 3u_{xxy}+ 3u_{xyy}+ u_{yyy},\\
u_{www}&= u_{xxx}- 3u_{xxy}+ 3u_{xyy}- u_{yyy},
\end{align*}
are uniformly bounded in two rectangular neighborhoods of the origin minus the origin, respectively, where the bounds only depend on $\epsilon, \eta_1, \eta_2, \eta_3, \eta_4, ||g||_{C^4}$.

So we derive the bounds of $u_{xxy}, u_{xyy}$ in a neighborhood of the origin, since
\begin{align}\begin{split}\label{eq-UZZZ}
u_{xxy}&= \frac{u_{zzz}-u_{www}-2u_{yyy}}{6},\\
u_{xyy}&=\frac{u_{zzz}+u_{www}-2u_{xxx}}{6}.\end{split}
\end{align}

We derived that $u$ has uniform $C^3$ estimates except at the origin.
Then, by a basic argument in calculus, we derive that $u\in C^{2,1}$ near the origin.
\end{proof}
\bigskip

\section{Proof of corrollaries}\label{sec-Cor}

\begin{proof}[Proof of Corollary \ref{cor-Mean}]
By Theorem \ref{thm-Main}, $u$ is $C^{2,1}$ in $x, y$. By the assumption, $\Delta u > C_0 > 0$ around the origin but not necessarily at the origin. Without loss of generality, under a rotation of coordinates, we assume for coordinates $\tilde{x}, \tilde{y}$,
\begin{align*}
u_{\tilde{x}\tilde{x}}(0,0)\geq C_0>0,\quad u_{\tilde{y}\tilde{y}}(0,0)=0
\end{align*}
which should hold for some $\tilde{x}, \tilde{y}$, since the Gauss curvature $k = 0$ at the origin.

We can rotate $x, y$ a little bit, such that none of the $x, y$ direction confirms with the
$\tilde{y}$ direction, and so
\begin{align*}
u_{xx} >C_1 >0,u_{yy} >C_1 >0,
\end{align*}
for some $C_1$ depends only on $C_0$ and the angle between the $x,y$ direction and the $\tilde{y}$ direction. This does not change that fact that $u_{xx}u_{yy} - u_{xy}^2 = 0$ at the origin, so we cannot apply the classic  theory for Monge-Amp\`ere equations.

Recall in the proof of Theorem \ref{thm-C21}, for the one dimensional model,  if $u_{xx} > C_1 > 0$, we derive $f = u^2_x = Mx^2 +Rx^3$, where $M > C_1^2$ is independent of $x$, and $R$ is smooth by the assumption that $f = g_{xx} -1$ is smooth. Then
\begin{align*}
u_x =sign(x)\cdot \sqrt{ Mx^2+Rx^3 }=x\sqrt{M+Rx},
\end{align*}
which has $C^k$ bounds for any integer $ k > 0$, which only depends on $M, k, ||f||_{C^{k+3}}$. In sum, $|D_x^k u| \leq  B_k(M,k,||f||_{C^{k+3}})$ for any $k$, and explicitly we have
\begin{align*}
D^2_x u(0) &=\sqrt{M},\\
D^3_x u(0) &= \frac{R(0)}{\sqrt{M}},\\
D^4_x u(0) &= \frac{12R_x(0)M - 3R(0)^2}{4M^\frac{3}{2}},\\
\cdots
\end{align*}

Then we check the two dimensional model, and derive $|D_x^k u| \leq  B_k(M,k,||f||_{C^{k+3}})$
around the origin. The estimates also hold for the two pairs of coordinate systems $z=\frac{x+y}{2}, w=\frac{x-y}{2}, z_1=\frac{x+2y}{5}, w_1=\frac{2x-y}{5}$, if none of these four coordinates points to the $\tilde{y}$ or $-\tilde{y}$ direction. We can rotate the $x, y$ coordinates system a little bit if one of them does.
Then
\begin{align*}
D_z^4 u &= u_{xxxx} + 4u_{xxxy} + 6u_{xxyy} + 4u_{xyyy} + u_{yyyy}\\
D_w^4 u &= u_{xxxx} -4u_{xxxy} + 6u_{xxyy}- 4u_{xyyy} + u_{yyyy}\\
D^4_{z_1} u &=u_{xxxx} +8u_{xxxy} +24u_{xxyy} +32u_{xyyy} +16u_{yyyy}
\end{align*}
are bounded, implying $u \in C^4$ near the origin.

Inductively, we can prove $u$ is $C^k$ near the origin by introducing more pairs of coordinate systems, and the corollary is verified.
\end{proof}

In Theorem \ref{thm-Main},
The forms of $\alpha, \beta$ that are allowed can be slightly generalized.
\begin{cor}\label{cor-Main}
Assume that $(B(O, r), g)$ satisfies the same assumption as in Theorem \ref{thm-Main}. 
If  a $C^{1,1}$ isometric embedding $X:(B(O, r), g) \rightarrow (\mathbb{R}^3, g_{can})$ which is of form \eqref{eq-GeneEmb} under local coordinates near $O$, satisfies the normalization \eqref{eq-GraphU}, and if $\alpha, \beta$ are $C^5$ in $x, y$, then $X\in C^{2,1}(B_g(O, r))$.
\end{cor}
\begin{proof}

Under the assumption of Corrollary \ref{cor-Main}, in any domain which does not include the origin, we have $ u \in C^{2, \mu}\,\, (0<\mu<1)$, and
\begin{align}\label{eq-UAB}
&u_{\alpha \alpha}>0, \,\, u_{\alpha \alpha}u_{\beta \beta}-u_{\alpha \beta}u_{\alpha \beta}>0,
\end{align}
 if we regard $u$ as a function of $\alpha, \beta$. Notice we does not necessarily have $u_{xx}>0$.

The system of equations
\begin{align*}
&\alpha_x^2+ \beta_x^2+ u_x^2= g_{xx},\\
&\alpha_x \alpha_y+ \beta_x \beta_y+ u_x u_y= g_{xy},\\
&\alpha_y^2+ \beta_y^2+ u_y^2= g_{yy},
\end{align*}
implies
\begin{align*}
& u_\alpha^2= g_{xx} x_\alpha^2+2g_{xy} x_\alpha y_\alpha+g_{yy} y_\alpha^2-1\\
& u_\beta^2= g_{xx} x_\beta^2+2g_{xy} x_\beta y_\beta+g_{yy} y_\beta^2-1\\
& u_\alpha u_\beta= g_{xx} x_\alpha y_\beta+g_{xy}( x_\beta y_\beta+x_\beta y_\alpha)+g_{yy} y_\alpha y_\beta, \\
\end{align*}
where the terms on the right hand side are $C^4$ in $\alpha, \beta$ by the assumption of Corrollary \ref{cor-Main}. Here $x, y$ can be regarded as $C^5$ functions of $\alpha, \beta$, since at the origin, we may choose $x, y$ as normal coordinates, and after a rotation, we assume $\alpha_y(0,0)=0, \beta_x(0,0)=0$. Then at the origin
\begin{align*}
&\alpha_x^2= g_{xx}=1,  \beta_y^2= g_{yy}=1.
\end{align*}
The Jacobian $\frac{\partial (\alpha, \beta)}{\partial (x, y)} =1$ at the origin. By the implicit function theorem, we can solve out $x, y$ as functions of $\alpha, \beta$.

We derive that $u_\alpha^2, u_\beta^2, u_\alpha u_\beta \in C^4$ and \eqref{eq-UAB} holds. We can apply the same method as in Section \ref{sec-TwoDim} to derive $u\in C^{2,1}$ near the origin. Then Corrollary \ref{cor-Main} follows.
\end{proof}

\begin{proof}[Proof of Corollary \ref{cor-MA}.] 
Without loss of generality, we assume \eqref{eq-GraphU} holds and $k=0$ at $O$.
Since $u= \Phi(r)$, we can check $\Phi\in C^{1,1}$, as $u\in C^{1,1}$.  We compute,
\begin{align*}
\det (D^2 u)= \frac{{\Phi_r}
\Phi_{rr}}{r}=k.
\end{align*}
So $k=\Psi(r)$, for some $C^{3}$ function $\Psi$ in $r$, by the assumption that $k$ is $C^3$ in $x, y$ and Lemma \ref{lem-rC3}. For $r<0$, we define $\Phi(r)= \Phi (-r), k(-r)=k(r)$. Then $\Phi\in C^{1,1}(\rho I), \Psi \in C^{2,1}(\rho I)$, since $\Phi(0)=\Phi_r(0)= k(0)= k_r(0)=0$. And when $r<0$, $ {\Phi_r}
\Phi_{rr}=\Psi r$ still holds. Notice $\Psi r $ is $C^3(\rho I)$.
Then
\begin{align*}
\Phi_r^2= \int_0^r \Psi(s) s ds,
\end{align*}
and $\Phi(0)= \Phi_r(0)=0$. Notice $ \int_0^r \Psi(s) s ds$ is in $C^4(\rho I)$ and $\frac{d}{d r} \int_0^r \Psi(s) s ds \cdot r= \Psi(r) r^2\geq 0$. Then we can apply \eqref{eq-C21S}, and derive
\begin{align*}
\Phi \in C^{2,1}(\frac{\rho}{2} I),
\end{align*}
which further implies $u$ is a $C^{2,1}$ function.
\end{proof}
\bigskip

\section{Nonnegative Gauss curvature case}\label{sec-NonNeg}
\bigskip
Set $\delta=\frac{r}{9}$ in this section. 
 Assume we have the Gauss curvature $k\geq 0$ on $(B(P, r), g)$. In addition, we assume $u$ is convex, which implies
\begin{align}\label{eq-XYZW}
u_{xx}, u_{yy}, u_{zz}, u_{ww}\geq 0,
\end{align}
(at where they exist) in $B_g(P, r)$. Here $(z, w)$ is the new coordinate system such that $z=\frac{x+y}{2}, w=\frac{x-y}{2}$. Note that convexity of $u$ is necessary, since we have an example
\begin{align*}
u= \ sign(x) \cdot (x^2+ |x| y^2),
\end{align*}
where $u_x^2$ and the metric of the graph are smooth, with nonnegative Gauss curvature, but $u$ is not $C^2$ with respect to $x$ on the $y$-axis. The main obstruction is that the graph of $u$ is not convex. 

For a point $(x_0, y_0) \in [-\delta, \delta]\times [-\delta, \delta]$, if $u_x(x_0, y_0)=0$ and
$u_x(x, y_0)<0$ for every $x\in (-\delta, x_0)$,
 we call $(x_0, y_0)$ a  left touch point of $u_x$. Notice the left touch point is unique for any $y\in [-\delta, \delta]$ if it exists, according to \eqref{eq-XYZW}. Similarly, we can define right touch points.

\begin{lem}\label{lem-Meas}
Assume $u(x,y)\in C^{1,1}(r I \times r I)$, $u_x^2 \in C^4(r I \times r I)$, $u_{xx}$ exists and $u_{xx}\geq 0$ in $\delta I\times \delta I$. Then in the square $[-\delta, \delta]\times [-\delta, \delta]$, the left and right touch points, form two sets of measure zero in $\mathbb{R}^2$.
\end{lem}
\begin{proof} Without loss of generality, we only consider set of the left touch points.
Define the left touch function $T_L(y)$ on $[-\delta, 
\delta]$, where $T_L(y_0)$ equals $x_0$ if we can find an $x_0$ such that $(x_0, y_0)$ is the left touch point on the line $y=y_0$, which is unique if any;  otherwise, $u_x(x, y_0)\geq 0$ for $x\in (-\delta, \delta)$, for which case we set $T_L(y_0)=-\delta$, or $u_x(x, y_0)< 0$ for $x\in (-\delta, \delta)$ for which case we set $T_L(y_0)=\delta$.

We check $T_L$ is a lower semi-continuous function. If $T_L(y)=-\delta$, then it's trivial since $T_L(x)\geq -\delta$.

 If $T_L(y)= x$ for some $x>-\delta$, then $(x, y)$ is a touch point, or $x=\delta$. In this case, for any $\epsilon\in (0, \delta+x)$, $u_x(x-\epsilon, y)<0$, so there is a neighborhood of $(x-\epsilon, y)$ in $\mathbb{R}^2$, such that $u_x<0$ in the neighborhood. So there is a $\eta>0$, for any $y_1\in (y-\eta, y+\eta)$,   $u_x(x-\epsilon, y_1)<0$, implying $T_L(y_1)>x-\epsilon$.

So $T_L$ is measurable as a lower semi-continuous function, and by the Fubini Theorem, its graph has measure zero.
\end{proof}

Lemma \ref{lem-Meas} shows that $u_{xxx}$ exists and is uniformly bounded in $\delta I \times \delta I$ minus the sets of left and right touch points, i.e. $u_{xxx}$ exists and is uniformly bounded, almost everywhere in $rI\times rI$.

We are ready to prove Theorem \ref{thm-Main1}, using mollifiers to help with applying \eqref{eq-UZZZ}.
\begin{proof}[proof of Theorem \ref{thm-Main1}]
Consider the $x$ direction first. 
Denote $r_0=6\delta, r_1=\sqrt{r^2-\delta^2}$. Then for any $x_0\in (-2\delta, 2\delta)$, we have
\begin{align}\begin{split}\label{eq-Radius}
-r_1<x_0-r_0< & x_0+r_0 <r_1,\\
x_0+\frac{r_0}{2}&\geq \delta,\\
x_0-\frac{r_0}{2}&\leq -\delta.
\end{split}
\end{align}

On each integral curve of $\frac{\partial}{\partial x}$ in $[-2\delta, 2\delta] \times \delta I$, we check whether $f=u_x^2$ has a zero.

 If not on the segment $[-2\delta, 2\delta]\times \{y_0\}$, we apply Corollary \ref{cor-MainEstm} to show $u$ is $C^{2,1}$ in $x$ on $[-\delta, \delta]\times \{y_0\}$  and \eqref{eq-C21S} holds with $s=\delta$. 
 
 If there is any zero on the segment $[-2\delta, 2\delta]\times \{y_0\}$, say $(x_0, y_0)$, then we consider  the interval $[x_0-r_0, x_0+r_0] \times \{y_0\} \subseteq B(P, r)$ by \eqref{eq-Radius}, and apply Theorem \ref{thm-C21} to show $u\in C^{2,1}$ in $x$ on the segment $[x_0-\frac{r_0}{2}, x_0+\frac{r_0}{2}] \times \{y_0\}$, which contains $[-\delta, \delta]\times \{y_0\}$ by \eqref{eq-Radius}. \eqref{eq-C21S} holds with $s=\frac{r_0}{2} >\delta$. 
 
 Hence $u_{xx}$ exists everywhere in $\delta I\times \delta I$ and is Lipschitz in $x$ on every integral curve of $\frac{\partial}{\partial x}$. In addition, by Lemma \ref{lem-Meas}, $u_{xxx}$ exists almost everywhere in $\delta I\times \delta I$, and has uniform bound \eqref{eq-C21S} with $s= \delta$. 
Consider the regulation of $u$ using the mollifier \eqref{eq-Rho},
 \begin{align*}
 u_\tau(x, y)= \tau^{-2} \rho(\frac{x}{\tau},  \frac{y}{\tau}) * u(x, y).
 \end{align*}
By Lemma \ref{lem-Smooth}, $|(u_\tau)_{xxx}| \leq C(r, g)$ in $B(P, \frac{\delta}{2})$, for any $\tau< dist( \partial (\delta I\times \delta I), \partial B(P, \frac{\delta}{2} ))$.

We have similar results in the $y, z, w$ direction. In the ball $B(P, \frac{\delta}{2})$, $(u_\tau)_{xxx}$, $(u_\tau)_{yyy}, $ $(u_\tau)_{zzz},$ $ (u_\tau)_{www}$ have a uniform bound which is independent of $\tau$. By \eqref{eq-UZZZ}, $(u_\tau)_{xxy}$, $(u_\tau)_{xyy},$ have a uniform bound as well. Now that $u_\tau$ has a uniform $C^3$ bound, so we can apply the Arzela-Ascoli Theorem to derive $u \in C^{2,1}(B(P, \frac{\delta}{2}))$.
\end{proof}

\appendix 

\section{Calculus lemmas}\label{Appen-CalculusL}
\begin{lem}\label{lem-FGC2}
Assume that $f$ satisfies conditions of Theorem \ref{thm-MainEstm}. Then
for $g=\frac{f}{x^2}, h=\frac{f^\prime}{x}$, we have $g, h \in C^2(I)$, and 
\begin{align*}
||g||_{C^2(I)}+ ||h||_{C^2(I)}\leq C ||f||_{C^4(I)},
\end{align*}
for some universal constant $C$.
\end{lem}
\begin{proof}
We can express
\begin{align*}
f(x)=\frac{ f^{\prime\prime}(0)}{2!} x^2 + \frac{ f^{\prime\prime \prime}(0)}{3!} x^3 + \int_0^x \int_0^{s_1} \int_0^{s_2} \int_0^{s_3} f^{(4)}(s_4) ds_4  ds_3  ds_2  ds_1.
\end{align*}
Then $g^{\prime\prime}$ includes terms like
\begin{align*}
&\frac{1}{x^2}  \int_0^{x} \int_0^{s_3} f^{(4)}(s_4) ds_4  ds_3 \\
&\frac{1}{x^3}  \int_0^{x} \int_0^{s_2} \int_0^{s_3} f^{(4)}(s_4) ds_4  ds_3  ds_2 \\
&\frac{1}{x^4} \int_0^x \int_0^{s_1} \int_0^{s_2} \int_0^{s_3} f^{(4)}(s_4) ds_4  ds_3  ds_2  ds_1
\end{align*}
which are all bounded by $||f||_{C^4(I)}$. In addition, as $x\rightarrow 0$, all these terms have limits by L'Hospital's rule, which shows $g^{\prime\prime}$ is continuous.

For $h$, the proof is similar. 

\end{proof}

\begin{lem}\label{lem-rC3}
If $k=k(x, y)$ lies in $C^3(B(O, \rho))$, and $k=\Psi(r)$, where $r=\sqrt{x^2+y^2}$, then $\Psi \in C^3([0, \rho))$.
\end{lem}
\begin{proof}
It follows directly from the fact that $\Psi(r)= k(r, 0)$.
\end{proof}

Denote $T=  \frac{\partial}{\partial x_1}$, a tangential vector on $\mathbb{R}^n$. Then it's integral curves are lines $(x_2, x_3, \cdots, x_n)= const$. Then we have the following lemma,
\begin{lem}\label{lem-WeakD}
Assume $w$ is a measurable function on $\Omega \subseteq\mathbb{R}^n$, and absolute continuous on every integral curves of $T$. In addition, $Tw$ exists almost everywhere, and it is integrable in $\Omega$. Then $Tw$ is a derivative of $w$ in the weak sense, i.e., for any smooth function $v$ which has compact support in $\Omega$,
\begin{align*}
\int_{\Omega}Tw\cdot v dx=- \int_{\Omega} w \cdot Tv dx
\end{align*}
\end{lem}
\begin{proof}
We need to show the one-dimensional case, then the general case is done by the Fubini's theorem. Denote $T= \frac{\partial}{\partial x}$. By the assumption, $w$ is a continuous function, and $Tw$ exists almost everywhere.

It's integration by parts. In fact,
\begin{align*}
T(wv)= Tw \cdot v+ w \cdot Tv.
\end{align*}
$wv$ is absolute continuous so we can integrate the equation using the fundamental theorem of calculus.
\end{proof}

Consider a mollifier, for $x\in \Omega \subseteq \mathbb{R}^n$,
\begin{align}\label{eq-Rho}
\rho(x)= c \exp\left( \frac{1}{|x|^2-1} \right),
\end{align}
when $|x|< 1$, and $\rho(x)=0$ when $|x|\geq 1$. Here $c$ is selected such that $\int_{\mathbb{R}^n} \rho(x) dx=1$.  For any $w \in L^1(\Omega)$ and $\tau>0$, the regulation of $w$ is defined to be
\begin{align*}
w_\tau(x)= \tau^{-n} \int_{\Omega} \rho \left(  \frac{x-y}{\tau}   \right) w(y)dy,
\end{align*}
where $\tau< dist(x, \partial \Omega)$. Then $w_\tau$ is a smooth function in a domain $\Omega^\prime$, if $\overline{\Omega^\prime} \subseteq \Omega$ and $\tau< dist(\partial \Omega^\prime, \partial \Omega)$.

\begin{lem}\label{lem-Smooth}
Assume that $w$ satisfies the assumption of Lemma \ref{lem-WeakD}. In addition, 
$|Tw|< C$ in $\Omega$. Then $|T w_\tau|< C$ in a domain $\Omega^\prime$, if $\overline{\Omega^\prime} \subseteq \Omega$ and $\tau< dist(\partial \Omega^\prime, \partial \Omega)$.
\end{lem}
\begin{proof}
\begin{align*}
T w_\tau(x)= \tau^{-n} \lim_{\epsilon\rightarrow 0} \int_{\Omega} \rho \left(  \frac{x}{\tau}   \right) \cdot \frac{w(x-y+\epsilon e_1)- w(x-y)}{\epsilon} dy,
\end{align*}
where $e_1=(1, 0, \cdots, 0) \in \mathbb{R}^n$.
Since $w$ is absolute continuous on the integral curves of $T$, 
\begin{align*}
\left\vert \frac{w(x-y+\epsilon e_1)- w(x-y)}{\epsilon} \right\vert = \left\vert \frac{1}{\epsilon}\int_0^\epsilon Tw(x-y+se_1)ds \right\vert<C.
\end{align*}
By the dominated convergence theorem,
\begin{align*}
Tw_\tau (x) = \tau^{-n} \int_{\Omega} \rho \left(  \frac{x}{\tau}   \right) \cdot Tw(x-y) dy,
\end{align*}
and has uniform bound $C$.
\end{proof}
If we reduce the absolute continuity condition to being continuous on integral curves of $T$, then the lemma is not right.
Cantor function in the one dimensional case is a counterexample.

\end{document}